\documentclass[12pt]{amsart}

\usepackage{enumitem}
\usepackage{psfrag}
\usepackage{graphicx}
\usepackage{pinlabel}
\usepackage{graphicx}
\usepackage{amsmath}
\usepackage{amssymb}
\usepackage{amscd}
\usepackage{autobreak}
\usepackage{picinpar}
\usepackage{times}
\usepackage{pb-diagram}
\usepackage{graphicx}
\usepackage{wrapfig}
\usepackage{xspace}
\usepackage{hyperref}
\usepackage{url}
\usepackage{caption}
\usepackage{subcaption}
\usepackage{fancyhdr}
\usepackage{overpic}
\usepackage{color}
\usepackage[square,comma,sort&compress,numbers]{natbib}
\usepackage{latexsym}
\usepackage{mathrsfs}
\usepackage{amsfonts}
\usepackage{hyperref}
\usepackage{amsthm}
\usepackage{appendix}
\usepackage{pdfsync}
\usepackage{float}

\theoremstyle{definition}
\newtheorem{theorem}{Theorem}[section]
\newtheorem{lemma}[theorem]{Lemma}

\newtheorem{definition}[theorem]{Definition}

\newtheorem{remark}[theorem]{Remark}

\newtheorem*{theorem*}{Theorem}
\graphicspath{{figures/}}   

\def\qed{\hfill{Q.E.D.}\smallskip}

\begin{document}

\title{\bf Generalized circle patterns on surfaces with cusps}
\author{Zhiwen Xiong, Xu Xu}
	
\date{\today}
	
\address{School of Mathematics and Statistics, Wuhan University, Wuhan, 430072, P.R.China}
\email{xiongzhiwen@whu.edu.cn}

\address{School of Mathematics and Statistics, Wuhan University, Wuhan, 430072, P.R.China}
\email{xuxu2@whu.edu.cn}
	
\thanks{MSC (2020): 52C25,52C26}
	
\keywords{Generalized circle patterens; Existence; Combinatorial curvature flow}
	
\begin{abstract}
In \cite{G-L}, Guo and Luo introduced generalized circle patterns on surfaces and proved their rigidity.
In this paper, we prove the existence of Guo-Luo's generalized circle patterns with prescribed generalized intersection angles on surfaces with cusps, which partially answers a question raised by Guo-Luo in \cite{G-L} and generalizes Bobenko-Springborn's hyperbolic circle patterns on closed surfaces to generalized hyperbolic circle patterns on surfaces with cusps. We further introduce the combinatorial Ricci flow and combinatorial Calabi flow for generalized circle patterns on surfaces with cusps, and prove the longtime existence and convergence of the solutions for these combinatorial curvature flows.

\end{abstract}
	
\maketitle
		
\section{Introduction}

\subsection{Generalized circle patterns on surfaces }
In \cite{B-S}, Bobenko and Springborn introduced a new type of circle patterns on closed surfaces, which is a generalization of Thurston's circle pattern in \cite{T1}.
In \cite{G-L}, Guo and Luo  further generalized Bobenko-Springborn's hyperbolic circle patterns on closed surfaces to eight other types of generalized circle patterns on surfaces and proved the rigidity of these generalized circle patterns.  We briefly recall Bobenko-Springborn's work \cite{B-S} and Guo-Luo's
work in \cite{G-L}. Let $(\Sigma,G)$ be a cellular decomposed surface with the set of vertices $V$, edges $E$ and faces $F$. For $v\in V$, $e\in E$ and $f\in F$, we use $v<e$ to denote that $v$ is an endpoint of $e$, and use $e<f$ to denote that $e$ is an edge of $f$.
Suppose $G^\ast$ is the dual cellular decomposition of $G$ and $V^\ast$ is the vertex set of $G^\ast$. Then each $f$ in $F$ contains exactly one vertex $f^\ast$ in $V^\ast$.
A generalized hyperbolic triangle $\triangle$ in $\mathbb{H}^2$ is a convex polygon bounded by three distinct geodesics $L_1$, $L_2$, $L_3$ and all other (if any) geodesics $L_{ij}$ perpendicular to both $L_i$ and $L_j$. In Guo-Luo's work \cite{G-L}, the type of a vertex $v$ is defined to be $1$, $0$ and $-1$ if $v\in \mathbb{H}^2$, $v\in \partial\mathbb{H}^2$ and $v\notin \mathbb{H}^2\cup\partial\mathbb{H}^2$, respectively. In this paper, we follow Guo-Luo to use this notation. In this way, one can classify the generalized hyperbolic triangles into ten types $(\epsilon_1,\epsilon_2,\epsilon_3)$,
where $\epsilon_i\in \{-1,0,1\}$. Suppose $L_i$ and $L_j$ are two adjacent edges of a vertex $v$. If $v$ is of type $-1$, denote $L_{ij}$ as the geodesic perpendicular to both $L_i$ and $L_j$.  For each vertex $v$ of the generalized hyperbolic triangle, let $B_v=v$ if $v$ is of type $1$, $B_v$ be a horocycle centered at $v$ if $v$ is of type $0$ and $B_v=L_{ij}$ if $v$ is of type $-1$. The generalized length of the edge $L_{i}\cap \triangle$ with vertices $v$, $u$ is the distance of $B_v$ and $B_u$ if $B_v\cap B_u=\emptyset$ and is the negative distance of $B_v\cap L_i$ and $B_u\cap L_i$ if $B_v\cap B_u\ne\emptyset$. The generalized angle $\beta_v$ at the vertex $v$ is defined as follows. If $v$ is of type $1$, $\beta_v$ equals to the intersection angle of $L_i$ and $L_j$;  if $v$ is of type $0$, $\beta_v$ equals to the arc length of the intersection of the horocycle $B_v$ bounded by $L_i$ and $L_j$; if $v$ is of type $-1$, $\beta_v$ equals to the distance between $L_i$ and $L_j$. Then we obtain a $(\epsilon_1,\epsilon_2,\epsilon_3)$-type generalized hyperbolic triangle with generalized edge lengths and generalized angles. The hyperbolic triangle in Bobenko-Springborn's construction on closed surfaces \cite{B-S} is of type $(1,1,1)$.
Replacing the hyperbolic triangle of type $(1,1,1)$ by the generalized hyperbolic triangle of type $(\epsilon,\epsilon,\delta)\in \{1,0,-1\}$,  Guo-Luo \cite{G-L} construct several new types of generalized circle patterns on polyhedral surfaces in the following way.

Set
\begin{eqnarray*}
\mathring{I}_\lambda=
\begin{cases}
\mathbb{R}_{>0} &\text{if}\ \lambda=-1,0,\\
(0,\pi) &\text{if}\ \lambda=1,
\end{cases} \quad
J_\mu=
\begin{cases}
\mathbb{R}_{>0} &\text{if}\ \mu=-1,1,\\
\mathbb{R}  &\text{if}\ \mu=0.
\end{cases}
\end{eqnarray*}
Given two functions $r: F\to J_{\epsilon\delta}$ and $\theta: E\to \mathring{I}_\delta$, for $e=f_1\cap f_2\in E$ with two endpoints $v_1,v_2\in V$,
one can construct a $(\epsilon,\epsilon,\delta)$ type generalized hyperbolic triangle $f^\ast_{1}v_1f^\ast_{2}$ with the generalized edge lengths $|f^\ast_{1}v_1|=r(f_{1})$, $|f^\ast_{2}v_1|=r(f_{2})$ and the generalized angle $\angle f^\ast_{1}v_1f^\ast_{2}=\theta(e)$ of type $\delta$.
Realize the hyperbolic  quadrilateral $f^\ast_{1}v_1f^\ast_{2}v_2$  as the isometric double of the generalized  hyperbolic triangle $f^\ast_{1}v_1f^\ast_{2}$ across the edge $f^\ast_{1}f^\ast_{2}$. For each edge $e\in E$, we construct a  hyperbolic quadrilateral in the same way. Then by isometrically gluing these  hyperbolic quadrilaterals, one obtains a $(\epsilon,\epsilon,\delta)$ type generalized circle patterns on polyhedral surfaces. We use $\beta_{(e,f_1)}$ to denote the generalized angle at $f^\ast_{1}$ in the generalized hyperbolic triangle $f^\ast_{1}v_1f^\ast_{2}$.
One can refer to Figure \ref{Figure 1} and Figure \ref{Figure 3} for an illustration of the angle $\beta_{(e,f_1)}$ in the $(1, 1, 0)$, $(1, 1, -1)$, $(0, 0, 1)$, $(0, 0, 0)$, $(0, 0, -1)$ type generalized hyperbolic triangles.
The generalized combinatorial curvature at $f^\ast$ is defined as
$K_f=\sum_{e;e<f}2\beta_{(e,f)}$.

Bobenko-Springborn \cite{B-S} proved the following existence and uniqueness result for the $(1,1,1)$ type hyperbolic circle patterns.

\begin{theorem}[\cite{B-S}, Theorem 3]
\label{Thm: BS}
Suppose $\Sigma$ is a cellular decomposed closed surface.
For any two functions $\theta: E\to (0,\pi)$ and $K: F\to (0,+\infty)$,
there exists a unique $(1,1,1)$ type hyperbolic circle pattern with the generalized angle $\theta$ and the generalized combinatorial curvature $K$ if and only if
\begin{equation*}
\sum_{f\in F^\prime} K_f
<\sum_{e\in E^\prime} 2(\pi-\theta(e)).
\end{equation*}
Here $F^\prime$ is a nonempty subset of $F$ and $E^\prime$ is the set of all edges incident to any face in $F^\prime$.
\end{theorem}

Bobenko-Springborn also proved the existence and uniqueness for the Euclidean circle patterns on closed surfaces in \cite{B-S}.
One can also refer to Guo \cite{Guo} for different proofs of  Theorem \ref{Thm: BS}.
Recently, Nie \cite{Nie} generalized Bobenko-Springborn's work to the spherical circle patterns on closed surfaces by introducing combinatorial  total geodesic curvature.

In \cite{G-L}, Guo-Luo proved the following rigidity result for the generalized circle patterns on surfaces.
\begin{theorem}[\cite{G-L}, Theorem 1.6]
\label{Thm: GL}
Under the setup of generalized circle pattern above, for any $(\epsilon,\epsilon,\delta)\in \{-1,0,1\}^3$ and $\theta: E\to \mathring{I}_\delta$,
the map $Q:J_{\epsilon\delta}^{|F|} \rightarrow \mathbb{R}^{|F|}$ sending $r$ to $K$ is a smooth embedding.
\end{theorem}

Guo-Luo (\cite{G-L} page 1277) further raised the question to investigate the image of the generalized combinatorial curvature $K$, which is equivalent to the existence of the generalized circle patterns on surfaces.
In this paper, we study the existence of generalized circle patterns of the type $(1, 1, 0)$, $(0, 0, 1)$, $(0, 0, 0)$ and $(0, 0, -1)$. Note that the surfaces obtained by gluing  generalized hyperbolic quadrilaterals of type $(1, 1, 0)$ are hyperbolic surfaces with cusps at $V$ and possibly conical points at $V^\ast$. The surfaces obtained by gluing generalized hyperbolic quadrilaterals of type $(0, 0, \delta)$ are hyperbolic surfaces with cusps at $V^\ast$ and possibly singularities at $V$. These singularities may be conical points, cusps or flares, corresponding to the case that $\delta$ equals to $1$, $0$ or $-1$, respectively. By using the continuity method, we give the following existence theorem for these generalized circle patterns.
It generalizes Bobenko-Springborn's results on hyperbolic circle patterns on closed surfaces in \cite{B-S} to generalized hyperbolic circle patterns on surfaces with cusps
and partially answers the question raised by Guo-Luo in \cite{G-L}.

\begin{theorem}\label{thm of existence}
Suppose $\Sigma$ is a cellular decomposed surface. Denote $K_{f_i}$ as $K_i$ for $i=1,...,|F|$.
\begin{enumerate}
\item [(1)] If $\epsilon=1$, $\delta=0$ and $\theta: E\to \mathbb{R}_{>0}$, then for any function $\hat{K}=(\hat{K}_1,...,\hat{K}_{|F|})\in \mathbb{R}^{|F|}_{>0}$,
there exists a unique $(1, 1, 0)$ type generalized circle pattern with radii $r=(r_{1},..., r_{|F|})\in \mathbb{R}^{|F|}$ such that
\begin{equation*}
	K_i(r)=\hat{K}_i , \quad \forall i\in \{1,...,|F|\},
\end{equation*}
if and only if
\begin{equation*}
	\sum_{f_i\in F^\prime} \hat{K_i}
	<2 |E^\prime| \pi.
\end{equation*}
Here $F^\prime$ is any nonempty subset of $F$ and $E^\prime$ is the set of all edges incident to any face in $F^\prime$.
\item [(2)] If $\epsilon=0$, $\delta \in \{1,-1,0\}$ and $\theta: E\to \mathring{I}_\delta$,
then for any function $\hat{K}=(\hat{K}_1,...,\hat{K}_{|F|})\in \mathbb{R}^{|F|}_{>0}$,
there exists a unique $(0, 0, \delta)$ type generalized circle pattern with radii $r=(r_{1},..., r_{|F|})\in \mathbb{R}^{|F|}$ such that
\begin{equation*}
	K_i(r)=\hat{K}_i , \quad \forall i\in \{1,...,|F|\},
\end{equation*}
if and only if $\hat{K}\in \mathbb{R}_{>0}^{|F|}$.
\end{enumerate}
\end{theorem}

\subsection{Combinatorial curvature flows for the generalized circle patterns}	
The combinatorial Ricci flow was first introduced by Chow-Luo \cite{C-L} to find Thurston's circle pattern with prescribed discrete Gaussian curvatures.
Subsequently, Guo-Luo \cite{G-L} introduced a discrete curvature flow for generalized circle patterns in the case that $(\epsilon,\epsilon,\delta)\in \{1,0,-1\}$ and obtained that the flow is a negative gradient flow of a strictly concave down function. The combinatorial Calabi flow was first introduced by Ge \cite{G} for Thurston's circle patterns on closed surfaces. In order to find a generalized circle pattern with prescribed generalized combinatorial curvature, we introduce the following combinatorial Ricci flow and combinatorial Calabi flow for the generalized circle patterns on a cellular decomposed surface with $\epsilon=1,\delta=0$ and $\epsilon=0,\delta \in \{-1,0,1\}$.

\begin{definition}\label{defn of comb flows}
Under the same assumptions as in Theorem \ref{thm of existence},
the combinatorial Ricci flow for the generalized circle patterns of type $(\epsilon,\epsilon,\delta)$ is defined to be
\begin{equation}\label{Eq: CRF}
\frac{\mathrm{d}r_i}{\mathrm{d}t}
=(K_i-\hat{K}_i)(\frac{1}{2}e^{r_i}-\frac{1}{2}\epsilon\delta e^{-r_i})
, \quad \forall i\in \{1,...,|F|\},
\end{equation}
and the combinatorial Calabi flow for the generalized circle patterns of type $(\epsilon,\epsilon,\delta)$ is defined to be
\begin{equation}\label{Eq: CCF}
\frac{\mathrm{d}r_{i}}{\mathrm{d}t}
=-\Delta(K-\hat{K})_i (\frac{1}{2}e^{r_i}-\frac{1}{2}\epsilon\delta e^{-r_i})
, \quad \forall i\in \{1,...,|F|\},
\end{equation}
where $\Delta=(\frac{\partial K_i}{\partial u_j})_{|F|\times |F|}$ is the discrete Laplace operator and $\delta=1,0,-1$.
\end{definition}

Note that the combinatorial Ricci flow (\ref{Eq: CRF}) introduced in Definition \ref{defn of comb flows}
is a modification of the discrete curvature flow
introduced by Guo-Luo for generalized circle patterns on surfaces in \cite{G-L}.
In this paper, we prove the following theorem for the combinatorial Ricci flow (\ref{Eq: CRF}) and the combinatorial Calabi flow (\ref{Eq: CCF}).


\begin{theorem}\label{flow}
Under the same assumptions as in Theorem \ref{thm of existence} (1),
the solutions of the combinatorial Ricci flow (\ref{Eq: CRF}) and the combinatorial Calabi flow (\ref{Eq: CCF}) for the $(1,1,0)$ type
generalized circle patterns exist for all time and converge exponentially fast for any initial data if and only if \begin{equation*}
	\sum_{f_i\in F^\prime} \hat{K_i}
	<2 |E^\prime| \pi.
\end{equation*}
Here $F^\prime$ is any nonempty subset of $F$ and $E^\prime$ is the set of all edges incident to any face in $F^\prime$.
\end{theorem}

\begin{theorem}\label{flow2}
Under the same assumptions as in Theorem \ref{thm of existence} (2),
the solutions of the combinatorial Ricci flow (\ref{Eq: CRF}) and the combinatorial Calabi flow (\ref{Eq: CCF})
for the $(0,0,\delta)$ type
generalized circle patterns exist for all time and converge exponentially fast for any initial data if and only if  $\hat{K}\in \mathbb{R}_{>0}^{|F|}$.
\end{theorem}

In \cite{Hazel}, Hazel considered the combinatorial Ricci flow in the cases of $(1, 1, 1)$,  $(1, 1, 0)$ and $(1, 1, -1)$, and proved the longtime existence and convergence of the solutions for the combinatorial Ricci flow with
$\hat{K}=(2\pi, ..., 2\pi)$. Theorem \ref{flow} generalizes Hazel's result for $(1, 1, 0)$ type generalized circle patterns with $\hat{K}=(2\pi, ..., 2\pi)$ to $(1, 1, 0)$ type generalized circle patterns with any prescribed combinatorial curvature.

\subsection{Organization of the paper}
In Section \ref{Generalized hyperbolic circle patterns in the case of (1,1,0)},
we study the generalized circle patterns in the case of $(1, 1, 0)$ and prove Theorem \ref{thm of existence} (1).
In Section \ref{Prescribed curvature flows in the case of (1,1,0)}, we prove Theorem \ref{flow}. In Section \ref{Generalized hyperbolic circle patterns in the case of (0,0,delta)},
we study the generalized circle patterns in the cases of $(0, 0, 1)$, $(0, 0, -1)$ and $(0, 0, 0)$ and prove Theorem \ref{thm of existence} (2). In Section \ref{Prescribed curvature flows in the case of (0,0,delta)}, we prove Theorem \ref{flow2}.
\\
\\
\textbf{Acknowledgements}\\[8pt]
The research of the authors is supported by National Natural Science Foundation of China
under grant no. 12471057.

\section{Generalized hyperbolic circle patterns in the case of (1, 1, 0)}\label{Generalized hyperbolic circle patterns in the case of (1,1,0)}

\subsection{The admissible space of generalized circle patterns in the case of (1, 1, 0)}\label{110}
In this section, we will characterize the existence of a $(1, 1, 0)$ type generalized hyperbolic triangle with two prescribed generalized edge lengths and the prescribed generalized angle between the two edges.

According to Lemma 3.6 in \cite{G-L}, given $\theta\in (0,+\infty)$, for any $r_{1}$, $r_{2}\in \mathbb{R}$,
there exists a $(1, 1, 0)$ type generalized hyperbolic triangle with two edges of generalized lengths $r_1$ and $r_2$ so that the generalized angle between them is $\theta$ of type $0$. Denote the generalized length of the edge opposite to $\theta$ by $l_{12}$ and the generalized angles opposite to $ r_{2}, r_{1}$ by $\beta_{1},\beta_{2}$, respectively.
Please refer to Figure \ref{Figure 1} (left).
Guo-Luo (\cite{G-L} Appendix A) gave the following formulas of cosine and sine laws for the $(1, 1, 0)$ type generalized hyperbolic triangle:
\begin{equation}\label{110cosbeta1}
\cos\beta_1=\frac{-e^{r_{2}}+e^{r_1}\cosh l_{12}}{e^{r_1}\sinh l_{12}},
\end{equation}
\begin{equation*}
\cos\beta_2=\frac{-e^{r_{1}}+e^{r_2}\cosh l_{12}}{e^{r_2}\sinh l_{12}},
\end{equation*}
\begin{equation}\label{110chl12}
	\cosh l_{12}=\frac{\theta^2}{2}e^{r_1+r_2}+\cosh(r_1-r_2),
\end{equation}
\begin{equation}\label{110shl12}
	\sinh l_{12}=\frac{\theta e^{r_{2}}}{\sin\beta_{1}}=\frac{\theta e^{r_{1}}}{\sin\beta_{2}}.
\end{equation}

\begin{figure}[htbp]
	\centering
	\includegraphics[scale=1.4]{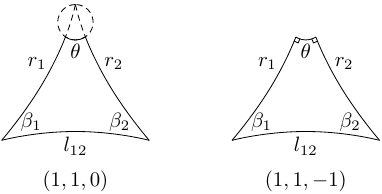}
	\caption{Generalized hyperbolic triangles for $\epsilon=1$, $\delta\in \{0,-1\}$}
	\label{Figure 1}
\end{figure}

By a change of variables
$$u_{i}=-2e^{-r_i},$$
the admissible space of a (1, 1, 0) type generalized hyperbolic quadrilateral in the parameter $u$ is $U_{e}= (-\infty,0)^{2}$.
Suppose $U$ is the admissible space of generalized circle patterns in the case of (1, 1, 0).
As a result, $U=\cap_{e\in E}U_{e}=(-\infty,0)^{|F|}$ is a convex polyhedron.
Furthermore, Guo-Luo proved the following result.

\begin{lemma}(\cite{G-L})\label{Lem: matrix negative 2}
	For a fixed generalized angle $\theta\in (0,+\infty)$,
	the matrix $\frac{\partial (\beta_1, \beta_2)}{\partial (u_1, u_2)}$ is symmetric and negative definite on $U_{e}$.
	As a result, the matrix $\frac{\partial (K_1,...,  K_{|F|})}{\partial(u_1,...,u_{|F|})}$ is symmetric and negative definite on $U$.
\end{lemma}

Hence for each $e\in E$, the following potential function
\begin{equation*}
	\mathcal{E}_{e}(u_{1},u_{2}) =\int^{(u_{1},u_{2})}(-\beta_1)\mathrm{d}u_{1}
	+(-\beta_2)\mathrm{d}u_{2}
\end{equation*}
is well defined on $\mathbb{R}^2_{<0}$.
By Lemma \ref{Lem: matrix negative 2},
$\mathcal{E}_{e}(u_{1},u_{2})$ is strictly convex on $U_{e}$ with $\nabla_{u_i}\mathcal{E}_{e}(u_{1},u_{2})=-\beta_i$ for $i=1,2$.
Furthermore, the function
\begin{equation*}
	\mathcal{E}(u)=\sum_{e=p\cap w\in E}2\mathcal{E}_{e}(u_{p}, u_{w})
\end{equation*}
is strictly convex on $U$ with $\nabla_{u_i}\mathcal{E}(u)=-\sum_{e;e<f_i}2\beta_{(e,f_i)}
=-K_i$ for $i\in\{1,...,|F|\}$.

\subsection{Limit behaviors of the generalized angles}
In this section, we consider the limit behavior of the generalized hyperbolic circle pattern in the case of $(1, 1, 0)$, in order to solve the existence problem.
\begin{lemma}\label{limit 2}
	Given $\theta\in (0,+\infty)$,
	for the $(1,1,0)$ type generalized hyperbolic triangle shown in Figure \ref{Figure 1} (left),
	\begin{description}
	\item[(1)] if $u_{i}$ converges to $0$ , then $\beta_{i}$ converges to $ 0$, where $i=1,2$;
	\item[(2)] if $u_{i}$ converges to $-\infty$ and $u_j$ converges to $t\in (-\infty,0)$, then $\beta_{i}$ converges to $\pi$ and $\beta_{j}$ converges to $0$, where $i=1,2$;
	\item[(3)] if both $u_{i}$ and $u_j$ converge to $-\infty$, then $\beta_{1}+\beta_{2}$ converges to $ \pi$, where $i=1,2$.
	\end{description}
\end{lemma}
\proof
\textbf{(1):}
Submitting (\ref{110chl12}) and the first equation of (\ref{110shl12}) into (\ref{110cosbeta1}) gives
\begin{equation}\label{110cotbeta1}
	\cot\beta_{1}
	=\frac{\theta^2+e^{-2r_2}-e^{-2r_1}}{2\theta e^{-r_1}} .
\end{equation}
Since $u_{i}=-2e^{-r_i}$, we have
\begin{equation}\label{110beta1}
\cot\beta_{1}
=\frac{-4\theta^2-u_2^2+u_1^2}{4\theta u_1}.
\end{equation}
Similarly, we have
\begin{equation}\label{110beta2}
\cot\beta_{2}
=\frac{-4\theta^2-u_1^2+u_2^2}{4\theta u_2}.
\end{equation}
Let $\{i,j\}=\{1,2\}$.
If $u_{i}$ converges to $0$, then
$$\lim_{u_{i} \to 0^-}\cot\beta_{i}=\lim_{u_{i} \to 0^-}\frac{-4\theta^2-u_j^2+u_i^2}{4\theta u_i}=\lim_{u_{i} \to 0^-}\frac{4\theta^2+u_j^2}{-4\theta u_i}=+\infty,$$
which implies that $\beta_{i}$ converges to 0.

\textbf{(2):}
If $u_{i}$ converges to $-\infty$ and $u_j$ converges to $t\in (-\infty,0)$, then
$$\lim_{(u_{i},u_j) \to (-\infty,t)}\cot\beta_{i} =\lim_{(u_{i},u_j) \to (-\infty,t)}\frac{-4\theta^2-u_j^2+u_i^2}{4\theta u_i}
=-\infty$$
and		
$$\lim_{(u_{i},u_j) \to (-\infty,t)}\cot\beta_{j} =\lim_{(u_{i},u_j) \to (-\infty,t)}\frac{-4\theta^2-u_i^2+u_j^2}{4\theta u_j}
=+\infty,$$		
which implies that $\beta_{i}$ converges to $\pi$ and $\beta_{j}$ converges to 0.

\textbf{(3):}
Submitting (\ref{110beta1}) and (\ref{110beta2}) into $\cot(\beta_1+\beta_2)=\frac{\cot \beta_1\cot \beta_2-1}{\cot \beta_1+\cot \beta_2}$ gives
\begin{equation*}
\cot(\beta_1+\beta_2)=\frac{-\theta^4+(u_1+u_2)^2(u_1-u_2)^2+16\theta^2u_1u_2}{4\theta (u_1+u_2)(\theta^2+(u_1-u_2)^2)}.
\end{equation*}
If both $u_{i}$ and $u_{j}$ converge to $-\infty$, then
\begin{equation*}
\begin{aligned}
&\lim_{(u_1,u_2)\to (-\infty,-\infty)}\cot(\beta_1+\beta_2)\\
=&\lim_{(u_1,u_2)\to (-\infty,-\infty)}\frac{1}{4\theta}[(u_1+u_2)\frac{(u_1-u_2)^2}{\theta^2+(u_1-u_2)^2}+(\frac{16u_1u_2}{u_1+u_2})\frac{\theta^2}{\theta^2+(u_1-u_2)^2}]. \end{aligned}
\end{equation*}
Note that \begin{equation*}
\begin{aligned}
&\min\{u_1+u_2,\frac{16u_1u_2}{u_1+u_2}\}\\
\le&(u_1+u_2)\frac{(u_1-u_2)^2}{\theta^2+(u_1-u_2)^2}+(\frac{16u_1u_2}{u_1+u_2})\frac{\theta^2}{\theta^2+(u_1-u_2)^2}\\
\le& \max\{u_1+u_2,\frac{16u_1u_2}{u_1+u_2}\}.
\end{aligned}
\end{equation*}
Since $$\lim_{(u_1,u_2)\to (-\infty,-\infty)}u_1+u_2=-\infty$$ and $$\lim_{(u_1,u_2)\to (-\infty,-\infty)}\frac{16u_1u_2}{u_1+u_2}=\lim_{(u_1,u_2)\to (-\infty,-\infty)}\frac{16}{\frac{1}{u_1}+\frac{1}{u_2}}=-\infty,$$
we have \[ \lim_{(u_1,u_2)\to (-\infty,-\infty)}(u_1+u_2)\frac{(u_1-u_2)^2}{\theta^2+(u_1-u_2)^2}+(\frac{16u_1u_2}{u_1+u_2})\frac{\theta^2}{\theta^2+(u_1-u_2)^2}=-\infty. \]
Therefore, $$\lim_{(u_1,u_2)\to (-\infty,-\infty)}\cot(\beta_1+\beta_2)=-\infty,$$
which implies that $\beta_{1}+\beta_{2}$ converges to $ \pi$ as $\beta_1$ and $\beta_2$ both converge to $-\infty$.
\qed

We further have the following variational formula for the generalized angles in a $(1,1,0)$ type generalized hyperbolic triangle. Such variational formulae were first obtained by Glickenstein-Thomas \cite{GT} for discrete conformal structures on closed surfaces.
\begin{lemma}\label{betaibetaj}
Given $\theta\in (0,+\infty)$,
for the $(1,1,0)$ type generalized hyperbolic triangle shown in Figure \ref{Figure 1} (left), we have \[ \frac{\partial \beta_{1}}{\partial u_1}=-\cosh l_{12}\frac{\partial \beta_{1}}{\partial u_2}, \quad \frac{\partial \beta_{2}}{\partial u_2}=-\cosh l_{12}\frac{\partial \beta_{2}}{\partial u_1}. \]
\end{lemma}

\begin{proof}
Differentiating (\ref{110cotbeta1}) in $r_1$ and submitting the first equation of (\ref{110shl12}) into it, we have
\begin{equation*}
	\frac{\partial \beta_{1}}{\partial r_1}=-\frac{\theta e^{r_2}\cosh l_{12}}{\sinh^2l_{12}}.
\end{equation*}
Since $u_i=-2e^{-r_i}$, then
\begin{equation}\label{equ:u3}
	\frac{\partial \beta_{1}}{\partial u_1}=-\frac{\theta e^{r_1+r_2}\cosh l_{12}}{2\sinh^2l_{12}}.
\end{equation}

Differentiating (\ref{110cotbeta1}) in $r_2$ and submitting the first equation of (\ref{110shl12}) into it, we have
\begin{equation*}
	\frac{\partial \beta_{1}}{\partial r_2}=\frac{\theta e^{r_1}}{\sinh^2l_{12}}.
\end{equation*}
Since $u_i=-2e^{-r_i}$, then
\begin{equation}\label{equ:u4}
\frac{\partial \beta_{1}}{\partial u_2}=\frac{\theta e^{r_1+r_2}}{2\sinh^2l_{12}}.
\end{equation}

Combining (\ref{equ:u3}) and (\ref{equ:u4}) gives
\[ \frac{\partial \beta_{1}}{\partial u_1}=-\cosh l_{12}\frac{\partial \beta_{1}}{\partial u_2}. \]
It follows by symmetry that
\[ \frac{\partial \beta_{2}}{\partial u_2}=-\cosh l_{12}\frac{\partial \beta_{2}}{\partial u_1}. \]
\end{proof}

\subsection{Existence theorem}

We rewrite Theorem \ref{thm of existence} (2) as follows.

\begin{theorem}\label{existence 2}
	Suppose $\Sigma$ is a cellular decomposed surface.
	If $\epsilon=1$, $\delta=0$ and $\theta: E\to \mathbb{R}_{>0}$, then for any function $\hat{K}\in \mathbb{R}^{|F|}_{>0}$,
	there exists a unique radii $r\in \mathbb{R}^{|F|}$ such that
	\begin{equation*}
		K_i(r)=\hat{K}_i , \quad \forall i\in \{1,...,|F|\},
	\end{equation*}
	if and only if
	\begin{equation*}
		\sum_{f_i\in F^\prime} \hat{K_i}
		<2|E^\prime| \pi.
	\end{equation*}
	Here $F^\prime$ is any nonempty subset of $F$ and $E^\prime$ is the set of all edges incident to any face in $F^\prime$.
\end{theorem}
\proof
Define the map $Q=\nabla \mathcal{E}(u): U \to \mathbb{R}_{<0}^{|F|},  (u_{1},..., u_{|F|}) \longmapsto (-K_{1},... , -K_{|F|})$. According to Theorem \ref{Thm: GL}, $Q$ is a smooth embedding.
Define $\delta=\{-K\in \mathbb{R}^{|F|}_{<0}|\sum_{f_i\in F^\prime} K_i
<2|E^\prime| \pi, \text{for any nonempty subset $F^\prime$ of $F$}\}$.
Then $Q(U)\subseteq \delta$ and $Q(U)$ is open in $\delta$.
It remains to prove that $Q(U)$ is closed in $\delta$.

Assume $\{u^{m}\}^{+\infty}_{m=1}$ is a sequence of points in $U$
such that
$$\lim_{m \to +\infty}Q(u^{m})=-\lim_{m\rightarrow +\infty}K^{(m)}:=t\in \delta.$$
By taking a subsequence, say $\{u^{m}\}^{+\infty}_{m=1}$, we just need to prove that $\lim_{m \to +\infty}u^{m}:=s\in U$.
Otherwise, $\lim_{m \to +\infty}u^{m}\in \partial U$.
The boundary of the admissible space $U$ consists of the following two parts
\begin{enumerate}
	\item [(1)]
	$\partial_{0}U=\{u\in [-\infty,0]^{|F|}| \text{there exists at least one $i\in \{1,2,...,|F|\}$ such that $u_i=0$}\}$;
	\item [(2)]
	$\partial_{\infty}U=\{u\in [-\infty,0]^{|F|}| \text{there exists at least one $i\in \{1,2,...,|F|\}$ such that $u_i=-\infty$}\}$.
\end{enumerate}
For the case (1), by Lemma \ref{limit 2} (1), we have $K_{i}=\sum_{e;e<f_{i}}2\beta_{(e,f_{i})}$ converges to $0$.
This contradicts with $t\in \delta$.
For the case (2), if $s_i\ne0$  for any $i\in \{1,2,...,|F|\}$ but $ F_1=\{f_i\in F| s_i=-\infty\}\ne \emptyset $,
then  $$\sum_{f_i\in F_1} K_i(u^{m})=2\sum_{e=f_i\cap f_j\in E_1}[\beta_{(e,f_i)}(u^{m})+\beta_{(e,f_j)}(u^{m})]+2\sum_{e=f_i\cap f_j\in E_2}\beta_{(e,f_i)}(u^{m}),$$
where $E_1=\{e=f_i\cap f_j\in E|f_i,f_j\in F_1\}$ and $E_2=\{e=f_i\cap f_j\in E|f_i\in F_1, f_j\notin F_1\}$.

Taking the limit at both sides, by Lemma \ref{limit 2} (2) and (3) we have
\[ \lim_{m \to +\infty}\sum_{f_i\in F_1} K_i(u^m)=2|E_1|\pi+2|E_2|\pi=2|E(F_1)|\pi .\]
Here $E(F_1)$ is the set of all edges incident to any face in $F_1$.
This contradicts with $t\in \delta$.
Thus we finish the proof.
\qed

\begin{remark}
According to Lemma 3.6 in \cite{G-L}, given $\theta\in (0,+\infty)$, for any $r_1,r_2\in \mathbb{R}_{>0}$, there exists a $(1,1,-1)$ type generalized hyperbolic triangle with two edges of generalized lengths $r_1$ and $r_2$ so that the generalized angle between them is $\theta$ of type $-1$. Please refer to Figure \ref{Figure 1} (right). Then we can construct a generalized circle patterns of type $(1,1,-1)$. Note that rigidity result holds for generalized circle patterns of type $(1,1,-1)$ by Theorem \ref{Thm: GL}, but
we can not get the existence result. Specially, in this case, we cannot depict the image of the generalized combinatorial curvature $K$.  Since in a single $(1,1,-1)$ type generalized triangle, the admissible space of $(\beta_{1},\beta_{2})$ is a subset of $(0,\pi)^2$ with the boundary given by $\cos\beta_{2}=-\cosh\theta\cos\beta_{1}$ and $\cos\beta_{2}=-\frac{\cos\beta_{1}}{\cosh\theta}$, which is not linear in $\beta_1$ and $\beta_2$. As a result, we are unable to prove the existence theorem for the type of $(1,1,-1)$.
\end{remark}

\section{Prescribed curvature flows in the case of (1, 1, 0)}\label{Prescribed curvature flows in the case of (1,1,0)}
In this section, we introduce two combinatorial curvature flows for the generalized hyperbolic circle patterns on surfaces of type $(1, 1, 0)$ and prove the longtime existence and convergence of the solutions for these combinatorial curvature flows.

Given
$\theta: E\to (0,+\infty)$ and $\hat{K}: F\to  \mathbb{R}_{>0}$, we define the following two energy functions
\begin{equation*}
	\tilde{\mathcal{E}}(u)
	=\mathcal{E}(u)+\sum_{i=1}^{|F|}\hat{K}_iu_i
	=-\int_0^u\sum_{i=1}^{|F|}(K_i-\hat{K}_i)\mathrm{d}u_i
\end{equation*}
and
\begin{equation*}
	\mathcal{C}(u)=\frac{1}{2}||K-\hat{K}||^{2}
	=\frac{1}{2}\sum_{i=1}^{|F|}(K_i-\hat{K_i})^{2} .
\end{equation*}
The function $\tilde{\mathcal{E}}(u)$ is sometimes referred as the combinatorial Ricci energy, and the function $\mathcal{C}(u)$ is
referred as the combinatorial Calabi energy.
It is easy to check that $\tilde{\mathcal{E}}(u)$ is a strictly convex function on $U$ with $\nabla_{u_i} \tilde{\mathcal{E}}(u)=\hat{K}_i-K_i$.
Furthermore, $\nabla_{u_i} \mathcal{C}(u)=\sum_{j=1}^{|F|}\frac{\partial K_j}{\partial u_i}(K_j-\hat{K_j})=(\Delta(K-\hat{K}))_i$.

\subsection{The longtime behavior of the combinatorial Ricci flow}
By a change of variables, i.e., $u_{i}=-2e^{-r_{i}}$, the equation (\ref{Eq: CRF}) can be reformulated as
\begin{eqnarray}\label{flow:u4}
	\frac{\mathrm{d}u_i}{\mathrm{d}t}=K_i-\hat{K}_i.
\end{eqnarray}
Hence, the combinatorial Ricci flow (\ref{flow:u4}) is a negative gradient flow of $\tilde{\mathcal{E}}(u)$.
Furthermore,
\begin{equation*}
	\frac{\mathrm{d}\tilde{\mathcal{E}}(u(t))}{\mathrm{d}t}
	=<\nabla\tilde{\mathcal{E}}(u(t)),\frac{\mathrm{d}u(t)}{\mathrm{d}t}>
	=-\left \| \nabla\tilde{\mathcal{E}}(u(t)) \right \| ^{2}\le 0,
\end{equation*}
which implies that $\tilde{\mathcal{E}}(u)$ is decreasing along the combinatorial Ricci flow (\ref{flow:u4}).
Similarly, since $(\frac{\partial K_i}{\partial u_j})_{|F|\times |F|}$ is negative definite on $U$ by Lemma \ref{Lem: matrix negative 2}, then
\begin{equation*}
	\frac{\mathrm{d}  \mathcal{C}(u(t))}{\mathrm{d} t}= \sum_{i,j=1}^{|F|}\frac{\partial \mathcal{C}}{\partial K_i} \frac{\partial K_i}{\partial u_j}\frac{\mathrm{d} u_j}{\mathrm{d} t}=
	\sum_{i,j=1}^{|F|}(K_i-\hat{K_i})\frac{\partial K_i}{\partial u_j}(K_j-\hat{K_j})\le 0 .
\end{equation*}
Thus $\mathcal{C}(u)$ is also decreasing along the combinatorial Ricci flow (\ref{flow:u4}).

We rewrite Theorem \ref{flow} for the combinatorial Ricci flow as follows.
\begin{theorem}\label{Thm: CRF110}
	Under the same assumptions as those in Theorem \ref{thm of existence} (1),
	the solution of the combinatorial Ricci flow (\ref{flow:u4}) for the generalized circle patterns of type $(1,1,0)$  exists for all time and converges exponentially fast for any initial data if and only if \begin{equation*}
		\sum_{f_i\in F^\prime} \hat{K_i}
		<2|E^\prime| \pi.
	\end{equation*}
	Here $F^\prime$ is any nonempty subset of $F$ and $E^\prime$ is the set of all edges incident to any face in $F^\prime$.
\end{theorem}

\proof
Suppose that the solution $u(t)$ of the combinatorial Ricci flow (\ref{flow:u4}) for the generalized circle patterns of type $(1,1,0)$  exists for all time and converges to $u^\ast\in U$, i.e., $\lim_{t \to +\infty}u(t)=u^{\ast}$. Then by Theorem \ref{thm of existence}, there exists a unique generalized circle pattern of type $(1,1,0)$ with $K^{\ast}=K(u^{\ast})\in \mathbb{R}^{|F|}_{>0}$ satisfying
\begin{equation*}
	\sum_{f_i\in F^\prime} K_i^\ast
	<2|E^\prime| \pi,
\end{equation*}
for any nonempty subset  $F^\prime\subset F$.
Choosing a sequence of time $\{t_{n}=n\}_{n=1}^{+\infty}$,
we have
\begin{equation*}
	u_i(n+1)-u_i(n)=u_i^\prime(\eta_{n})
	=K_i(u(\eta_{n}))-\hat{K}_i, \eta_{n}\in (n,n+1).
\end{equation*}
Taking the limit on both sides gives
\begin{equation*}
	\lim_{n \to +\infty}(K_i(u(\eta_{n}))-\hat{K}_i)
	=\lim_{n \to +\infty}(u_i(n+1)-u_i(n))=0.
\end{equation*}
Since $K_i(u)$ is a continuous function of $u$, we have $\hat{K}_i=\lim_{n \to +\infty}K_i(u(\eta_{n}))=K_i(u^{\ast})=K_i^{\ast}$ for all $i\in \{1,...,|F|\}$, which implies that $\sum_{f_i\in F^\prime} \hat{K_i}
<2|E^\prime| \pi$ holds for any nonempty subset  $F^\prime\subset F$.

Conversely, if $\sum_{f_i\in F^\prime} \hat{K_i}
<2|E^\prime| \pi$ holds for any nonempty subset  $F^\prime\subset F$,
then by Theorem \ref{thm of existence},
there exists a unique $\tilde{u}\in U$ such that $K_i(\tilde{u})=\hat{K}_i$ for all $i\in \{1,...,|F|\}$.
Therefore, $\tilde{\mathcal{E}}(u)$ has a critical point $\tilde{u}$ and is bounded from below.
Since $\tilde{\mathcal{E}}(u)$ is decreasing along the combinatorial Ricci flow (\ref{flow:u4}),
then the solution $u(t)$ is contained in a compact subset of $\mathbb{R}^{|F|}$.
We claim that $u(t)$ is contained in a compact subset of $U=(-\infty,0)^{|F|}$.
We shall prove the theorem assuming the claim and then prove the claim.

Choosing a sequence of time $\{t_{n}=n\}_{n=1}^{+\infty}$, we have \[ \tilde{\mathcal{E}}(u(n+1))-\tilde{\mathcal{E}}(u(n))
=\tilde{\mathcal{E}}^\prime(u(t))|_{\eta_{n}}
=<\nabla\tilde{\mathcal{E}}, \frac{\mathrm{d} u}{\mathrm{d} t}>|_{\eta_{n}}
=-\left \| \nabla\tilde{\mathcal{E}} \right \|^{2}|_{\eta_{n}},\ \eta_{n}\in (n,n+1) .\]
Then taking the limit on both sides gives
\begin{equation*}
	\lim_{n \to +\infty}(-\left \| \nabla\tilde{\mathcal{E}} \right \|^{2}|_{\eta_{n}})
	=\lim_{n \to +\infty}(\tilde{\mathcal{E}}(u(n+1))
	-\tilde{\mathcal{E}}(u(n)))=0,
\end{equation*}
which implies that $\lim_{n \to +\infty}K_i(u(\eta_{n}))=\hat{K}_i$ for all $i\in \{1,...,|F|\}$.
Since $u(t)$ is contained in a compact subset of $U$, there exists a subsequence of $\{\eta_{n}\}_{n=1}^{+\infty}$, still denoted as $\{\eta_{n}\}_{n=1}^{+\infty}$, such that $\lim_{n \to +\infty}u(\eta_{n})=\tilde{u}^\prime$.
This implies
$K_f(\tilde{u}^\prime)=\lim_{n\rightarrow +\infty}K_f(u(\eta_{n}))=K_f(\tilde{u})$.
This further implies $\tilde{u}=\tilde{u}^\prime$ by Theorem \ref{Thm: GL}.
Set $\Gamma(u)=K-\hat{K}$, then the matrix $D\Gamma|_{u=\tilde{u}}<0$  by Lemma \ref{Lem: matrix negative 1}, i.e., $D\Gamma|_{u=\tilde{u}}$ has $|F|$ negative eigenvalues.
This implies that $\tilde{u}$ is a local attractor of the combinatorial Ricci flow (\ref{flow:u4}).
The conclusion follows from Lyapunov Stability Theorem (\cite{Pontryagin}, Chapter 5).

Now we prove the claim.
The boundary of $U=(-\infty,0)^{|F|}$ consists of the following two parts
\begin{enumerate}
	\item [(1)]
	$\partial_{\infty}U=\{u\in [-\infty,0]^{|F|}| \text{there exists at least one $i\in \{1,2,...,|F|\}$ such that $u_i=-\infty$}\}$;
	\item [(2)]
	$\partial_{0}U=\{u\in [-\infty,0]^{|F|}| \text{there exists at least one $i\in \{1,2,...,|F|\}$ such that $u_i=0$}\}$.
\end{enumerate}
The solution $u(t)$ can not reach the boundary $\partial_{\infty}U$
since $u(t)$ is contained in a compact subset of $\mathbb{R}^{|F|}$. Then we just need to consider the second part of the boundary.

Suppose the solution $u(t)$ of the combinatorial Ricci flow ($\ref{flow:u4}$) can reach the boundary $\partial_{0}U $.
Then there exists some time $t'\in (0,+\infty]$ and $i\in \{1,2,...,|F|\}$ such that $\lim_{t\rightarrow t'}u_i(t)=0$.
By Lemma \ref{limit 2} (1), we have $\lim_{t\rightarrow t'}\beta_{(e,f_i)}(u_i(t))=0$.
There exists $c\in (u_i(0),0)$, such that whenever $u_i(t)>c$,
the angle $\beta_{(e,f_i)}(u_i(t))$ is smaller than $\frac{\hat{K_i}}{2n}$ for each $e<f_i$, where $n$ is the number of edges contained in face $f_i$. This implies $K_i(u_i(t))<\hat{K_i}$ for any $u_i(t)>c$.
Choose a time $t_0\in (0,t')$ such that $u_i(t_0) > c$, which can be done because $\lim_{t\to t'}u_i(t)=0$. Set $a=\inf\{t<t_0|u_i(s)>c,\forall s\in(t,t_0]\}$, and then $u_i(a)= c$. Note that $\frac{\mathrm{d} u_i}{\mathrm{d} t}=K_i-\hat{K_i}<0$ on $(a,t_0)$, we have $u_i(t_0)<u_i(a)=c$, which contradicts with $u_i(t_0) > c $.
Therefore, the solution $u(t)$ of the combinatorial Ricci flow ($\ref{flow:u4}$) can not reach the boundary $\partial_{0}U $.
\qed

\subsection{The longtime behavior of the combinatorial Calabi flow}
By a change of variables, i.e., $u_{i}=-2e^{-r_{i}}$, the equation (\ref{Eq: CCF}) can be reformulated as
\begin{equation}\label{flow:u3}
	\frac{\mathrm{d}u_{i}}{\mathrm{d}t}=
	-\Delta(K-\hat{K})_i.
\end{equation}
Hence, the combinatorial Calabi flow (\ref{flow:u3}) is a negative gradient flow of $\mathcal{C}(u)$.
Furthermore,
\begin{equation*}
	\frac{\mathrm{d}  \mathcal{C}(u(t))}{\mathrm{d} t}=<\nabla\mathcal{C}(u(t)),\frac{\mathrm{d}u(t)}{\mathrm{d}t}>
	=-\left \| \nabla\mathcal{C}(u(t)) \right \| ^{2}\le 0,
\end{equation*}
which implies that $\mathcal{C}(u)$ is decreasing along the combinatorial Calabi flow (\ref{flow:u3}).
Similarly, since $(\frac{\partial K_i}{\partial u_j})_{|F|\times |F|}$ is negative definite on $U$ by Lemma \ref{Lem: matrix negative 2}, then
\begin{equation*}
	\frac{\mathrm{d}\tilde{\mathcal{E}}(u(t))}{\mathrm{d}t}
	=\sum_{i=1}^{|F|}\frac{\partial \tilde{\mathcal{E}}}{\partial u_i}\frac{\mathrm{d} u_i}{\mathrm{d} t}  =
	\sum_{i,j=1}^{|F|}(K_i-\hat{K_i})\frac{\partial K_i}{\partial u_j}(K_j-\hat{K_j}) \le 0.
\end{equation*}
Thus $\tilde{\mathcal{E}}(u)$ is also decreasing along the combinatorial Calabi flow (\ref{flow:u3}).

We rewrite Theorem \ref{flow} for the combinatorial Calabi flow as follows.
\begin{theorem}\label{Thm: CCF110}
	Under the same assumptions as in Theorem \ref{thm of existence},
	the solution of the combinatorial Calabi flow (\ref{flow:u3}) for the generalized circle patterns of type $(1,1,0)$ exists for all time and converges exponentially fast for any initial data if and only if \begin{equation*}
		\sum_{f_i\in F^\prime} \hat{K_i}
		<2|E^\prime| \pi.
	\end{equation*}
	Here $F^\prime$ is any nonempty subset of $F$ and $E^\prime$ is the set of all edges incident to any face in $F^\prime$.
\end{theorem}
\proof
Suppose that the solution $u(t)$ of the combinatorial Calabi flow (\ref{flow:u3}) for the generalized circle patterns of type $(1,1,0)$  exists for all time and converges to $u^\ast\in U$, i.e., $\lim_{t \to +\infty}u(t)=u^{\ast}$. Then by Theorem \ref{thm of existence}, there exists a unique generalized circle pattern of type $(1,1,0)$ with $K^{\ast}=K(u^{\ast})\in \mathbb{R}^{|F|}_{>0}$ satisfying
\begin{equation*}
	\sum_{f_i\in F^\prime} K_i^\ast
	<2|E^\prime| \pi,
\end{equation*}
for any nonempty subset  $F^\prime\subset F$.
Choosing a sequence of time $\{t_{n}=n\}_{n=1}^{+\infty}$, we have \[ u_{i}(n+1)-u_{i}(n)=u_{i}'(\eta_{n})=-\Delta(K(u(\eta_{n}))-\hat{K})_i, \eta_{n}\in (n,n+1) .\]
Taking the limit on both sides gives
\[ \lim_{n \to +\infty}-\Delta(K(u(\eta_{n}))-\hat{K})_i=
\lim_{n \to +\infty}(u_{i}(n+1)-u_{i}(n))=0 .
\]
Since $\Delta$ is negative definite on $U$ by Lemma \ref{Lem: matrix negative 2},
then we have $\lim_{n \to +\infty} (K_i(u(\eta_{n}))-\hat{K}_i)=0$, which yields that $\sum_{f_i\in F^\prime} \hat{K_i}
<2|E^\prime| \pi$ holds for any nonempty subset  $F^\prime\subset F$.

Conversely, if $\sum_{f_i\in F^\prime} \hat{K_i}
<2|E^\prime| \pi$ for any nonempty subset $F^\prime$,
then by Theorem \ref{thm of existence},
there exists a unique $\tilde{u}\in U$ such that $K_i(\tilde{u})=\hat{K}_i$ for all $i\in \{1,2,...,|F|\}$.
Therefore, $\tilde{\mathcal{E}}(u)$ has a critical point $\tilde{u}$ and is bounded from below.
Since $\tilde{\mathcal{E}}(u)$ is decreasing along the combinatorial Calabi flow (\ref{flow:u3}),
then the solution $u(t)$ is contained in a compact subset of $\mathbb{R}^{|F|}$. We also claim that $u(t)$ is contained in a compact subset of $U=(-\infty,0)^{|F|}$. We shall prove the theorem assuming the claim and then prove the claim.

Note that the discrete Laplace operator $\Delta$ is strictly negative by Lemma \ref{Lem: matrix negative 2}. By the continuity of the eigenvalues of $\Delta$, there exists $\lambda_0>0$ such that each eigenvalue of $\Delta$ is less than $-\sqrt{\lambda_0}$ along the combinatorial Calabi flow (\ref{flow:u3}). Then along the combinatorial Calabi flow (\ref{flow:u3}), we have
\begin{equation*}
	\frac{\mathrm{d}  \mathcal{C}(u(t))}{\mathrm{d} t}=\sum_{i=1}^{|F|}\frac{\partial \mathcal{C}}{\partial u_i}\frac{\mathrm{d} u_i}{\mathrm{d} t} =-(K-\hat{K})^{\mathrm{T}}\Delta^2 (K-\hat{K})\le -2\lambda_0 \mathcal{C}(u(t)),
\end{equation*}
which implies that $\mathcal{C}(u(t))=\frac{1}{2}||K(t)-\hat{K}||^{2}\le e^{-2\lambda_0 t}||K(0)-\hat{K}||^{2}$.
According to Theorem \ref{Thm: GL},
we have $||u(t)-\tilde{u}||^2\le C_1||K(t)-\hat{K}||^{2}\le C_2e^{-2\lambda_0 t}$ for some positive constants $C_1$, $C_2$. This completes the proof.

Now we prove the claim.
The boundary of $U$ consists of the following two parts
\begin{enumerate}
	\item [(1)]
	$\partial_{\infty}U=\{u\in [-\infty,0]^{|F|}| \text{there exists at least one $i\in \{1,2,...,|F|\}$ such that $u_i=-\infty$}\}$;
	\item [(2)]
	$\partial_{0}U=\{u\in [-\infty,0]^{|F|}| \text{there exists at least one $i\in \{1,2,...,|F|\}$ such that $u_i=0$}\}$.
\end{enumerate}
The solution $u(t)$ can not reach the boundary $\partial_{\infty}U$ since $u(t)$ is contained in a compact subset of $\mathbb{R}^{|F|}$. Then we just need to consider the second part of the boundary.

Suppose the solution $u(t)$ of the combinatorial Calabi flow ($\ref{flow:u3}$) can reach the boundary $\partial_{0}U $.
Then there exists some time $t'\in (0,+\infty]$ and $i\in \{1,2,...,|F|\}$ such that $\lim_{t\rightarrow t'}u_i(t)=0$.
According to Lemma \ref{betaibetaj}, for $e=f_i\cap f_j\in E$, we have
\[ \frac{\partial \beta_{(e,f_i)}}{\partial u_i}=-\cosh l_{ij}\frac{\partial \beta_{(e,f_i)}}{\partial u_j}. \]
Note that
\begin{align*}
	\frac{\mathrm{d}u_i}{\mathrm{d}t}&=-\Delta(K-\hat{K})_i\\&=-\frac{\partial K_i}{\partial u_i}(K_i-\hat{K_i})-\sum_{j\ne i}\frac{\partial K_i}{\partial u_j}(K_j-\hat{K_j})\\
	&=-2\sum_{e=f_i\cap f_j\in E}\frac{\partial \beta_{(e,f_i)}}{\partial u_i}(K_i-\hat{K_i})-2\sum_{e=f_i\cap f_j\in E}\frac{\partial \beta_{(e,f_i)}}{\partial u_j}(K_j-\hat{K_j})\\
	&=2\sum_{e=f_i\cap f_j\in E}\frac{\partial \beta_{(e,f_i)}}{\partial u_j}[\cosh l_{ij}(K_i-\hat{K_i})-(K_j-\hat{K_j})] ,
\end{align*}

By Lemma \ref{limit 2} (1), there exists a number $M_1\in (-\infty,0)$ such that if $u_i\ge M_1$, then $K_i-\hat{K_i}<0$.
Suppose $u_i(t)\rightarrow 0$ and then $r_i(t)\rightarrow +\infty$.
Thus,
\begin{align*}
	\cosh l_{ij}
	&=\frac{\theta(e)^2}{2}e^{r_i+r_j}+\cosh(r_i-r_j) \\
	&=e^{r_i}\left(\frac{\theta(e)^2}{2}e^{r_j}+\frac{1}{2}e^{-r_j}+\frac{1}{2}e^{r_j-2r_i}\right)\\
	&\ge e^{r_i} \cdot \min\{1,\frac{\theta(e)^2}{2}\}\cdot\cosh r_j\\
	&\rightarrow+\infty.
\end{align*}
Therefore, as $u_i(t)\rightarrow 0$, $\cosh l_{ij}$ tends to $+\infty$ uniformly. Since $\hat{K_j}$ is given and $K_j>0$,
then there exists a number $M_2\in (M_1,0)$ such that if $u_i\ge M_2$, then
\begin{equation*}
	\cosh l_{ij}(K_i-\hat{K_i})-(K_j-\hat{K_j})<0
\end{equation*}
for each $e=f_i\cap f_j\in E$.
By (\ref{equ:u4}), we have $\frac{\partial \beta_{(e,f_i)}}{\partial u_j}>0$.
Hence, if $u_i\ge M_2$, we have $\frac{\mathrm{d}u_i}{\mathrm{d}t}<0$.
The rest of the proof is the same as that of the combinatorial Ricci flow in Theorem \ref{Thm: CRF110},
we omit it here.
\qed

\section{Generalized hyperbolic circle patterns in the case of (0, 0, $\delta$)}\label{Generalized hyperbolic circle patterns in the case of (0,0,delta)}

\subsection{The admissible space of generalized circle patterns in the case of (0, 0, $\delta$)}\label{001}
In this section, we will characterize the existence of a (0, 0, $\delta$) type generalized hyperbolic triangle with two prescribed generalized edge lengths and the prescribed generalized angle between the two edges.

According to Lemma 3.6 in \cite{G-L}, given $\theta\in \mathring{I}_\delta$, for any $r_{1}$, $r_{2}\in \mathbb{R}$,
there exists a $(0, 0, \delta)$ type generalized hyperbolic triangle with two edges of generalized lengths $r_1$ and $r_2$ so that the generalized angle between them is $\theta$ of type $\delta$. Denote the generalized length of the edge opposite to $\theta$ by $l_{12}$ and the generalized angles opposite to $ r_{2}, r_{1}$ by $\beta_{1},\beta_{2}$, respectively  .
Please refer to Figure \ref{Figure 3} for the cases of $\delta=1,-1,0$, respectively. Guo-Luo (\cite{G-L} Appendix A) gave the following formulas of cosine and sine laws for the $(0, 0, \delta)$ type generalized hyperbolic triangle:
\begin{description}
\item[(1)] for $\delta=1$,
\begin{equation}\label{001beta1}
\beta_1^2=\frac{e^{r_{2}}-e^{l_{12}-r_{1}}}{e^{r_{1}+l_{12}}},
\end{equation}
\begin{equation*}
\beta_2^2=\frac{e^{r_{1}}-e^{l_{12}-r_{2}}}{e^{r_{2}+l_{12}}},
\end{equation*}
\begin{equation}\label{001chl12}
e^ {l_{12}}=e^{r_1+r_2}\sin^2 \frac{\theta}{2},
\end{equation}
\begin{equation}\label{001shl12}
e^ {l_{12}}=\frac{e^{r_{2}}\sin \theta }{2\beta_{1}}=\frac{e^{r_{1}}\sin \theta }{2\beta_{2}};
\end{equation}
\item[(2)] for $\delta=-1$,
\begin{equation}\label{00-1beta1}
\beta_1^2=\frac{e^{r_{2}}+e^{l_{12}-r_{1}}}{e^{r_{1}+l_{12}}},
\end{equation}
\begin{equation*}
\beta_2^2=\frac{e^{r_{1}}+e^{l_{12}-r_{2}}}{e^{r_{2}+l_{12}}},
\end{equation*}
\begin{equation}\label{00-1chl12}
e^ {l_{12}}=e^{r_1+r_2}\sinh^2 \frac{\theta}{2},
\end{equation}
\begin{equation}\label{00-1shl12}
e^ {l_{12}}=\frac{e^{r_{2}}\sinh \theta }{2\beta_{1}}=\frac{e^{r_{1}}\sinh \theta }{2\beta_{2}};
\end{equation}
\item[(3)] for $\delta=0$,
\begin{equation}\label{000beta1}
\beta_{1}^2=e^{r_2-r_1-l_{12}},
\end{equation}
\begin{equation*}
\beta_{2}^2=e^{r_1-r_2-l_{12}},
\end{equation*}
\begin{equation}\label{000shl12}
e^{l_{12}}=\frac{\theta e^{r_2} }{\beta_{1}}=\frac{\theta e^{r_1} }{\beta_{2}}.
\end{equation}
\end{description}

\begin{figure}[htbp]
	\centering
	\includegraphics[scale=1.1]{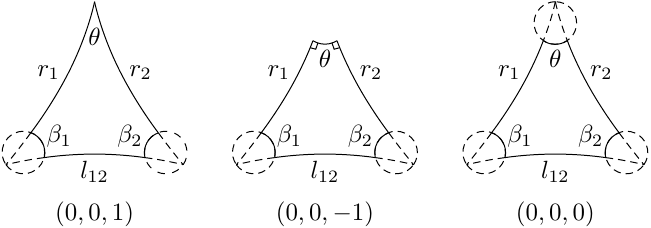}
	\caption{Generalized hyperbolic triangles for $\epsilon=0$, $\delta\in \{1,-1,0\}$}
	\label{Figure 3}
\end{figure}

By a change of variables
$$u_{i}=-2e^{-r_i},$$
the admissible space of a (0, 0, $\delta$) type generalized hyperbolic quadrilateral in the parameter $u$ is
$U_{e}= (-\infty,0)^{2}$.
Suppose $U$ is the admissible space of generalized circle patterns in the case of (0, 0, $\delta$).
As a result, $U=\cap_{e\in E}U_{e}=(-\infty,0)^{|F|}$ is a convex polyhedron.
Furthermore, Guo-Luo proved the following result.

\begin{lemma}(\cite{G-L})\label{Lem: matrix negative 1}
	For a fixed generalized angle $\theta\in \mathring{I}_\delta$,
	the matrix $\frac{\partial (\beta_1, \beta_2)}{\partial (u_1, u_2)}$ is symmetric and negative definite on $U_{e}$.
	As a result, the matrix $\frac{\partial (K_1,...,  K_{|F|})}{\partial(u_1,...,u_{|F|})}$ is symmetric and negative definite on $U$.
\end{lemma}

Hence for each $e\in E$, the potential function
$\mathcal{E}_{e}(u_{1},u_{2}) =\int^{(u_{1},u_{2})}(-\beta_1)\mathrm{d}u_{1}
	+(-\beta_2)\mathrm{d}u_{2}$
is well defined on $\mathbb{R}^2_{<0}$.
By Lemma \ref{Lem: matrix negative 1},
$\mathcal{E}_{e}(u_{1},u_{2})$ is strictly convex on $U_{e}$ with $\nabla_{u_i}\mathcal{E}_{e}(u_{1},u_{2})=-\beta_i$ for $i=1,2$.
Furthermore, the function
\begin{equation*}
	\mathcal{E}(u)=\sum_{e=p\cap w\in E}2\mathcal{E}_{e}(u_{p}, u_{w})
\end{equation*}
is strictly convex on $U$ with $\nabla_{u_i}\mathcal{E}(u)=-\sum_{e;e<f_i}2\beta_{(e,f_i)}
=-K_i$ for $i\in\{1,...,|F|\}$.

\subsection{Limit behaviors of the generalized angles}
In this section, we consider the limit behavior of the generalized hyperbolic circle pattern in the case of (0, 0, $\delta$), in order to solve the existence problem. Here $\delta\in \{1,-1,0\}$.
\begin{lemma}\label{limit 1}
Given $\theta\in \mathring{I}_\delta$,
for the $(0,0,\delta)$ type generalized hyperbolic triangle,
\begin{description}
	\item[(1)] if $u_{i}$ converges to $-\infty$ , then $\beta_{i}$ converges to $ +\infty$, where $i=1,2$;
		
	\item[(2)] if $u_{i}$ converges to $0$ , then $\beta_{i}$ converges to $ 0$, where $i=1,2$.
\end{description}
\end{lemma}
\proof
For $\delta=1$,
submitting (\ref{001chl12}) and the first equation of (\ref{001shl12}) into (\ref{001beta1}) gives
\begin{equation*}
	\beta_{1}
	=e^{-r_1}\cot\frac{\theta}{2} .
\end{equation*}
Since $u_{i}=-2e^{-r_i}$, we have
$\beta_{1}
=-\frac{1}{2}u_1\cot\frac{\theta}{2}.$
Similarly, we have
$\beta_{2}
=-\frac{1}{2}u_2\cot\frac{\theta}{2}.$

For $\delta=-1$, submitting (\ref{00-1chl12}) and the first equation of (\ref{00-1shl12}) into (\ref{00-1beta1}) gives
\begin{equation*}
	\beta_{1}=e^{-r_1}\coth\frac{\theta}{2}.
\end{equation*} 	
Since $u_{i}=-2e^{-r_i}$, we have
$\beta_{1}=-\frac{1}{2}u_1\coth\frac{\theta}{2}.$
Similarly, we have
$\beta_{2}=-\frac{1}{2}u_2\coth\frac{\theta}{2}$.

For $\delta=0$, submitting the first equation of (\ref{000shl12}) into (\ref{000beta1}) gives
\begin{equation*}
	\beta_{1}
	=e^{-r_1}\theta^{-1}.
\end{equation*}
Since $u_{i}=-2e^{-r_{i}}$, we have
$\beta_1=-\frac{1}{2}u_1\theta^{-1}.$
Similarly, we have
$\beta_2=-\frac{1}{2}u_2\theta^{-1}.$

In conclusion,
\begin{equation}\label{equ:betai}
\beta_{i}=-\frac{1}{2}s(\theta)u_i
\end{equation}
for $i = 1$, $2$, where $s(\theta)$ equals to $\cot\frac{\theta}{2}$, $\coth\frac{\theta}{2}$ or $\theta^{-1} $ if $\delta=$ $1$, $-1$ or $0$ respectively. Note that $s(\theta)>0$.

\textbf{(1):}
If $u_{i}$ converges to $-\infty$ , then
$\lim_{u_{i} \to -\infty}\beta_{i} =-\lim_{u_{i} \to -\infty}\frac{1}{2}u_is(\theta)
=+\infty.$		

\textbf{(2):}
If $u_{i}$ converges to $0$ , then
$\lim_{u_{i} \to 0}\beta_{i}=-\lim_{u_{i} \to 0}\frac{1}{2}u_is(\theta)=0.$
\qed

\subsection{Existence theorem}

We rewrite Theorem \ref{thm of existence} (1) as follows.

\begin{theorem}\label{existence 1}
	Suppose $\Sigma$ is a cellular decomposed surface. If $\epsilon=0$, $\delta \in \{1,-1,0\}$ and $\theta: E\to \mathring{I}_\delta$,
	then for any function $\hat{K}\in \mathbb{R}^{|F|}_{>0}$,
	there exists a unique radii $r\in \mathbb{R}^{|F|}$ such that
	\begin{equation*}
		K_i(r)=\hat{K}_i , \quad \forall i\in \{1,...,|F|\},
	\end{equation*}
	if and only if $\hat{K}\in \mathbb{R}_{>0}^{|F|}$.
\end{theorem}
\proof
Define the map $Q=\nabla \mathcal{E}(u): U \to \mathbb{R}_{<0}^{|F|},  (u_{1},..., u_{|F|}) \longmapsto (-K_{1},... , -K_{|F|})$. Since
$$ -K_i(u)=-\sum_{e;e<f_i}2\beta_{(e,f_i)}=\sum_{e;e<f_i}s(\theta(e))u_i,$$
which is a linear function of $u_i$, then Q is a homeomorphism from $\mathbb{R}^{|F|}_{<0}$ to $\mathbb{R}^{|F|}_{<0} $.
Thus we finish the proof.
\qed

\section{Combinatorial curvature flows in the case of (0, 0, $\delta$)}\label{Prescribed curvature flows in the case of (0,0,delta)}
In this section, we introduce two combinatorial curvature flows for the generalized hyperbolic circle patterns on surfaces of type $(0, 0, \delta)$ and prove the longtime existence and convergence of the solutions for these combinatorial curvature flows. Given two functions
$\theta: E\to \mathring{I}_\delta$ and $\hat{K}: F\to  \mathbb{R}$, consider two functions $\tilde{\mathcal{E}}(u) $ and $\mathcal{C}(u)$ defined in section \ref{Prescribed curvature flows in the case of (1,1,0)}.
It is easy to check that $\tilde{\mathcal{E}}(u)$ is a strictly convex function on $U$ with $\nabla_{u_i} \tilde{\mathcal{E}}(u)=\hat{K}_i-K_i$.
Furthermore, $\nabla_{u_i} \mathcal{C}(u)=\Delta(K-\hat{K})_i$.

\subsection{The longtime behavior of the combinatorial Ricci flow}
By a change of variables, i.e., $u_{i}=-2e^{-r_{i}}$, (\ref{Eq: CRF}) can be reformulated as
\begin{eqnarray}\label{flow:u1}
	\frac{\mathrm{d}u_i}{\mathrm{d}t}=K_i-\hat{K}_i.
\end{eqnarray}
Hence, the combinatorial Ricci flow (\ref{flow:u1}) is a negative gradient flow of $\tilde{\mathcal{E}}(u)$.
Furthermore, follow the same steps as in section \ref{Prescribed curvature flows in the case of (1,1,0)}, we have that $\tilde{\mathcal{E}}(u)$ and $\mathcal{C}(u)$ are decreasing along the combinatorial Ricci flow (\ref{flow:u1}).

We rewrite Theorem \ref{flow2} for the combinatorial Ricci flow as follows.
\begin{theorem}\label{Thm: CRF}
	Under the same assumptions as those in Theorem \ref{thm of existence} (2),
	the solution of the combinatorial Ricci flow (\ref{flow:u1}) for the generalized circle patterns of type $(0,0,\delta)$  exists for all time and converges exponentially fast for any initial data if and only if $\hat{K}\in \mathbb{R}_{>0}^{|F|}$. Here $\delta$ $= 1$, $ 0$ or $-1$.
\end{theorem}

The proof of Theorem \ref{Thm: CRF} is almost the same as that of Theorem \ref{Thm: CRF110}, so we omit it here.
%
\subsection{The longtime behavior of the combinatorial Calabi flow}
By a change of variables, i.e., $u_{i}=-2e^{-r_{i}}$, (\ref{Eq: CCF}) can be reformulated as
\begin{equation}\label{flow:u2}
	\frac{\mathrm{d}u_{i}}{\mathrm{d}t}=
	-\Delta(K-\hat{K})_i.
\end{equation}
Hence, the combinatorial Calabi flow (\ref{flow:u2}) is a negative gradient flow of $\mathcal{C}(u)$.
Furthermore, follow the same steps as section \ref{Prescribed curvature flows in the case of (1,1,0)}, we have that $\tilde{\mathcal{E}}(u)$ and $\mathcal{C}(u)$ are decreasing along the combinatorial Calabi flow (\ref{flow:u2}).

We rewrite Theorem \ref{flow2} for the combinatorial Calabi flow as follows.
\begin{theorem}\label{Thm: CCF}
	Under the same assumptions as in Theorem \ref{thm of existence} (2),
	the solution of the combinatorial Calabi flow (\ref{flow:u2}) for the generalized circle patterns of type $(0,0,\delta)$ exists for all time and converges exponentially fast for any initial data if and only if $\hat{K}\in \mathbb{R}_{>0}^{|F|}$. Here $\delta$ = 1, 0 or -1.
\end{theorem}
\proof
The proof of necessity is the same as Theorem \ref{Thm: CCF110}, so we omit it here. If $\hat{K}\in \mathbb{R}_{>0}^{|F|}$,
then by Theorem \ref{thm of existence},
there exists a unique $\tilde{u}\in U$ such that $K_i(\tilde{u})=\hat{K}_i$ for all $i\in \{1,2,...,|F|\}$.
Therefore, $\tilde{\mathcal{E}}(u)$ has a critical point $\tilde{u}$ and is bounded from below.
Since $\tilde{\mathcal{E}}(u)$ is decreasing along the combinatorial Calabi flow (\ref{flow:u2}),
then the solution $u(t)$ is contained in a compact subset of $\mathbb{R}^{|F|}$. We also claim that $u(t)$ is contained in a compact subset of $U=(-\infty,0)^{|F|}$. Assuming the claim, then by the same discussion as in Theorem \ref{Thm: CCF110}, we prove the theorem.

Now we prove the claim.
The boundary of $U$ consists of the following two parts
\begin{enumerate}
	\item [(1)]
	$\partial_{\infty}U=\{u\in [-\infty,0]^{|F|}| \text{there exists at least one $i\in \{1,2,...,|F|\}$ such that $u_i=-\infty$}\}$;
	\item [(2)]
	$\partial_{0}U=\{u\in [-\infty,0]^{|F|}| \text{there exists at least one $i\in \{1,2,...,|F|\}$ such that $u_i=0$}\}$.
\end{enumerate}
Similar to the proof of Theorem \ref{Thm: CRF110},
the solution $u(t)$ can not reach the boundary $\partial_{\infty}U$. Then we just need to consider the second part of the boundary.

Suppose the solution $u(t)$ of the combinatorial Calabi flow ($\ref{flow:u2}$) can reach the boundary $\partial_{0}U $.
Then there exists some time $t'\in (0,+\infty]$ and $i\in \{1,2,...,|F|\}$ such that $\lim_{t\rightarrow t'}u_i(t)=0$.
According to (\ref{equ:betai}), for $e=f_i\cap f_j\in E$, we have
\begin{equation}\label{equ:beta}
	\beta_{(e,f_i)}=-\frac{1}{2}u_is(\theta(e)),
\end{equation}
where $s(\theta(e))$ equals to $\cot\frac{\theta(e)}{2}$, $\coth\frac{\theta(e)}{2}$, or $\theta(e)^{-1} $ if $\delta=$ $1$, $-1$ or $0$ respectively.

Differentiating (\ref{equ:beta}) in $u_i$ and $u_j$ gives
\begin{equation*}
	\frac{\partial \beta_{(e,f_i)}}{\partial u_i}=-\frac{1}{2}s(\theta(e)), \quad
	\frac{\partial \beta_{(e,f_i)}}{\partial u_j}=0.
\end{equation*}

Note that
\begin{align*}
	\frac{\mathrm{d}u_i}{\mathrm{d}t}&=-\Delta(K-\hat{K})_i\\&=-\frac{\partial K_i}{\partial u_i}(K_i-\hat{K_i})-\sum_{j\ne i}\frac{\partial K_i}{\partial u_j}(K_j-\hat{K_j})\\
	&=-2\sum_{e=f_i\cap f_j\in E}\frac{\partial \beta_{(e,f_i)}}{\partial u_i}(K_i-\hat{K_i})-2\sum_{e=f_i\cap f_j\in E}\frac{\partial \beta_{(e,f_i)}}{\partial u_j}(K_j-\hat{K_j})\\
	&=\sum_{e=f_i\cap f_j\in E}s(\theta(e))(K_i-\hat{K_i}).
\end{align*}
The rest of the proof is the same as that of the combinatorial Ricci flow in Theorem \ref{Thm: CRF110},
we omit it here.
\qed

\end{document}